\RequirePackage[l2tabu, orthodox]{nag}
\documentclass[11pt]{amsart} %
\author[A. Hammerlindl]{Andy Hammerlindl}
\address{School of Mathematical Sciences, Monash University, Victoria 3800 Australia} \urladdr{ http://users.monash.edu.au/~ahammerl/}  \email{andy.hammerlindl@monash.edu}

\author[R. Potrie]{Rafael Potrie}
\address{CMAT, Facultad de Ciencias, Universidad de la Rep\'ublica, Uruguay}
\urladdr{www.cmat.edu.uy/$\sim$rpotrie}\email{rpotrie@cmat.edu.uy}

\title{Classification of systems with center-stable tori}
\thanks{R.P.~was partially supported by CSIC group 618 and by MathAmSud-Physeco.}

\usepackage{amsmath}
\usepackage{amssymb}
\usepackage{amsfonts}
\usepackage{amsthm}
\usepackage{amscd}
\usepackage{graphicx}
\usepackage{setspace}
\usepackage{cleveref}
\usepackage{fourier}
\usepackage[T1]{fontenc}

\makeatletter
\def\saveenum{\xdef\@savedenum{\the\c@enumi\relax}}
\def\resetenum{\global\c@enumi\@savedenum}
\makeatother

\newcommand{\HHU}{Rodriguez Hertz, Rodriguez Hertz, and Ures}

\newcommand{\WLOG}{Without loss of generality}

\newcommand{\bbR}{\mathbb{R}}

\newcommand{\bbZ}{\mathbb{Z}}

\newcommand{\bbT}{\mathbb{T}}

\newcommand{\Es}{E^s}
\newcommand{\Ec}{E^c}
\newcommand{\Eu}{E^u}
\newcommand{\Ecu}{E^{cu}}
\newcommand{\Ecs}{E^{cs}}

\newcommand{\Ws}{W^s}
\newcommand{\Wc}{W^c}
\newcommand{\Wu}{W^u}

\newcommand{\inv}{^{-1}}

\newcommand{\cF}{\mathcal{F}}

\newcommand{\Fcu}{\mathcal{F}^{cu}}
\newcommand{\Fcs}{\mathcal{F}^{cs}}

\newcommand{\Ku}{K^u}

\newcommand{\Ks}{K^s}

\newcommand{\Om}{\Omega}

\newcommand{\piA}{\pi^A}

\newcommand{\del}{\partial}
\newcommand{\delcs}{\partial^{cs} \Om}

\newcommand{\without}{\setminus}

\newcommand{\Omin}{\Om^{\circ}}

\newcommand{\Lcs}{L^{cs}}
\newcommand{\Lcu}{L^{cu}}

\newcommand{\tM}{\tilde{M}}

\newcommand{\ep}{\epsilon}
\newcommand{\lam}{\lambda}

\newcommand{\gam}{\gamma}

\newcommand{\qandq}{\quad \text{and} \quad}

\newcommand{\dist}{\operatorname{dist}}
\newcommand{\length}{\operatorname{length}}
\newcommand{\volume}{\operatorname{volume}}

\newcommand{\id}{\operatorname{id}}

\newcommand{\pis}{\pi^s}
\newcommand{\piu}{\pi^u}

\newcommand{\Hu}{H^u}
\newcommand{\Hs}{H^s}

\newcommand{\cT}{\mathcal{T}}

\newcommand{\subof}{\subset}
\newcommand{\sans}{\setminus}
\newcommand{\ti}{\times}

\numberwithin{equation}{section}
\newtheorem{thm}[equation]{Theorem}
\newtheorem{cor}[equation]{Corollary}
\newtheorem{lemma}[equation]{Lemma}
\newtheorem{prop}[equation]{Proposition}

\theoremstyle{remark}
\newtheorem*{remark} {\textbf{Remark}}

\providecommand{\acknowledgement}{{\noindent \textbf{Acknowledgements.}}\quad}

\begin{document}

\maketitle

\begin{abstract}
    This paper gives a classification of partially hyperbolic systems
    in dimension 3 which have at least one torus tangent to the
    center-stable bundle.
\end{abstract}
\section{Introduction} %

A long-standing question in the study of partially hyperbolic dynamical
systems was whether a system with one-dimensional center
possessed a foliation tangent to that center direction.
This question was recently answered by \HHU{},
who constructed a partially hyperbolic diffeomorphism on the 3-torus
without such a center foliation \cite{rhrhu2016nondyn}.
Crucial to their construction is a 2-torus embedded in the manifold
tangent to the center-stable direction.
In this paper, we give a classification of all
3-dimensional
partially hyperbolic systems with center-stable tori.

Previous classification results 
relied on the notion of a leaf conjugacy between the center foliations
of two different partially hyperbolic systems \cite{hp20XXsurvey}.
In certain manifolds, such as the 3-torus, the presence of
a center-stable or center-unstable torus is the only potential obstruction
to having an invariant center foliation \cite{pot2015partial}.
In the current setting,
the lack of a center foliation in general
means that it is not possible to use
a global leaf conjugacy to classify the dynamics.
Instead, we first remove all of the center-stable and center-unstable
tori from the system
leaving dynamics defined on an open manifold.
Looking at the dynamics on each of the connected components,
we show that it has the form of a topological skew product.

\medskip{}

Before giving the full result, we state the definitions of partial
hyperbolicity and related concepts.
A diffeomorphism $f$ of a closed connected manifold $M$
is \emph{partially hyperbolic}
if there is a splitting of the tangent bundle
\[
    TM = \Es \oplus \Ec \oplus \Eu
\]
such that each subbundle is non-zero and
invariant under the derivative $Df$
and
\[
    \|Df v^s\| < \|Df v^c\| < \|Df v^u\|
    \qandq
    \|Df v^s\| < 1 < \|Df v^u\|
      \]
for all $x  \in  M$ and unit vectors
$v^s  \in  \Es(x)$, 
$v^c  \in  \Ec(x)$, and
$v^u  \in  \Eu(x)$.
There exist unique foliations $\Ws$ and $\Wu$ tangent to $\Es$ and $\Eu$,
but in general there does not exist a foliation tangent to $\Ec$.
A \emph{center-stable torus} is an embedded torus
tangent to $\Ecs = \Ec \oplus \Es$, and 
a \emph{center-unstable torus} is an embedded torus
tangent to $\Ecu = \Ec \oplus \Eu$.
We also refer to these objects as $cs$ and $cu$-tori.

The definition of partial hyperbolicity above is sometimes
called \emph{pointwise} partial hyperbolicity,
in comparison to a stricter condition called
\emph{absolute} partial hyperbolicity.
In dimensional three, it is not possible for
an absolutely partially hyperbolic system to have
a $cs$ or $cu$-torus \cite{BBI2, HP2}.
Therefore, this paper only uses the pointwise definition of partial
hyperbolicity.

\medskip{}

\begin{thm} \label{thm:main}
    Suppose $f$ is a partially hyperbolic diffeomorphism of a closed, oriented
    3-manifold $M$ which has at least one center-stable or center-unstable
    torus.
    Then,
    \begin{enumerate}
        \item there is a finite and pairwise disjoint collection
        $\{ T_1, T_2, \ldots, T_n\}$
        of all center-stable and center-unstable tori,

        \item every connected component $U_i$ of
        $M \sans (T_1 \cup \ldots \cup T_n)$
        is homeomorphic to $\bbT^2 \ti (0,1)$,

        \item there is $k  \ge  1$ such that
        $f^k$ maps each $T_i$ to itself and each $U_i$ to itself,
        and

        \item for each $U_i \subof M$,
        there is an embedding $h:U_i \to \bbT^2 \ti \bbR$
        such that the homeomorphism
        $h \circ f^k \circ h \inv$ from $h(U_i)$ to itself
        is of the form
        \[            h f^k h \inv(v, s) = (A(v), \phi(v,s) )  \]
        where A : $\bbT^2  \to  \bbT^2$ is a hyperbolic toral automorphism
        and $\phi:h(U_i) \to \bbR$ is continuous.
        Moreover,
        if $v  \in  \bbT^2$, then 
        \begin{enumerate}
            \item $h \inv(v \ti \bbR)$
            is a curve tangent to $\Ec_f$,

            \item $h \inv(\Ws_A(v) \ti \bbR)$
            is a surface tangent to $\Ecs_f$, and

            \item $h \inv(\Wu_A(v) \ti \bbR)$
            is a surface tangent to $\Ecu_f$.
        \end{enumerate}  \end{enumerate}  \end{thm}
The first three items of the theorem state previously known results.
The proof of item (1) is given in \cite{ham20XXprox} and
item (2) is a restatement of the work of
Rodriguez Hertz, Rodriguez Hertz, and Ures
to classify which 3-manifolds allow tori with hyperbolic dynamics
\cite{rhrhu2011tori}.
Item (3) follows as a consequence of items (1) and (2).
Therefore, the substance and novelty of \cref{thm:main} lies in item (4).

The theorem shows that there is a form of conjugacy between the
dynamics on a region $U_i$, and a skew product over an Anosov map on $\bbT^2$.
However, this is complicated by the fact that
the dynamics on the
$cs$ and $cu$-tori may not be Anosov in general
and could contain sinks or sources \cite{ham20XXconstructing}.
In such cases, the boundary components of $h(U_i)$ will not be the graphs of
continuous functions from $\bbT^2$ to $\bbR$.
The map $h$ might be thought of as a ``ragged leaf conjugacy''
between $f$ and $A \ti \id$
since $h$ maps the smooth subset $U_i$ of $M$ to a subset of $\bbT^2 \ti \bbR$
with a ragged boundary.
See figure \ref{fig:ragged}.
\begin{figure}
   \includegraphics{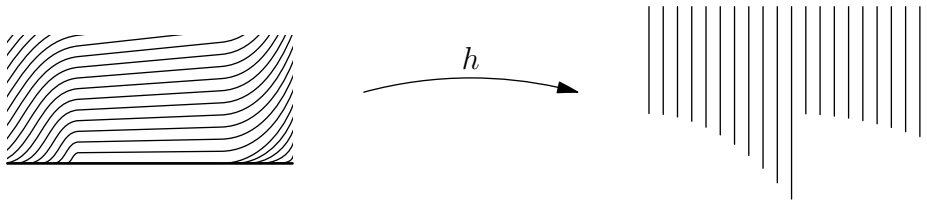} 
   \caption[Ragged]{
   The ``ragged leaf conjugacy'' given by \cref{thm:main}.
   The left side of the figure depicts curves tangent to the center
   direction $\Ec$ in a region $U_i$ as they approach a $cs$-torus
   which has a sink.
   This $cs$-torus is
   depicted by a thick line at the bottom of the left side.
   Each center curve is mapped by $h$ to a vertical segment in $\bbT^2 \ti \bbR$
   and these segments are shown at right.
   As a consequence of results in \cref{sec:fibers},
   the center curves must ``skip over'' the basin of the sink
   as they approach the $cs$-torus, and
   for a point $x  \in  U_i$, the distance along the center curve
   from $x$ to the torus is discontinuous in $x$.
   As a result, the lengths of the vertical segments are also
   discontinuous.
   }
   \label{fig:ragged}
\end{figure}
Nevertheless, the center direction $\Ec$ may still be accurately described in a
neighbourhood of a $cs$ or $cu$-torus.  See \cref{sec:fibers} for details.

Also note that the curves and surfaces given in item (4)
are incomplete with respect to the Riemannian metric induced from $M$.

At a coarse level,
the main steps of proving \cref{thm:main} are similar those
of previous
classification results \cite{ham2013leaf,ham-nil,HP1}.
(See also the recent survey \cite{hp20XXsurvey}.)
By the work of Brin, Burago, and Ivanov, there are branching foliations
tangent to $\Ecs$ and $\Ecu$ on $M$ \cite{BI,BBI2}.
We restrict these branching foliations to one of the components $U_i$
and consider the center curves given by intersecting leaves of the
two branching foliations.
By analyzing the interaction of the dynamics with the
branching foliations,
we show that these center curves correspond to fibers of the semiconjugacy
given by Franks \cite{Franks1}.
Then, using an averaging method along center leaves, we construct the
function $h$.

Two major complications to applying these steps in the current context
are that $U_i$ is not compact, and that the
leaves of the branching foliations have tangencies with the boundary of
$U_i$.  To handle these complications,
we consider the dynamics and the branching foliations
both on the closure of $U_i$ and on compact subsets in the interior of $U_i$.
We also lift these compact subsets to the universal cover and mainly do
analysis there.

To begin, \cref{sec:dimtwo} gives a detailed description of the 2-dimensional
dynamics possible on a $cs$ or $cu$-torus.
\Cref{sec:bran} introduces branching foliations and
states a number of properties which hold for all partially hyperbolic
systems in dimension three.
\Cref{sec:defOm} states properties specific to systems containing a $cs$ or
$cu$-torus and introduces a number of important propositions
which are then proved in
sections \ref{sec:csfiber}, \ref{sec:hiddentorus}, and \ref{sec:fibers}.
\Cref{sec:leafconj} then uses these to prove \cref{thm:main}.
As part of the overall proof, we need a result on the structure
of branching foliations on $\bbT^2 \ti [0,1]$
and this is given in an appendix.

\section{Dynamics in dimension two} \label{sec:dimtwo} %

In order to understand 3-dimensional systems with $cs$ and $cu$-tori,
it is necessary to fully understand the 2-dimensional dynamics acting
on these tori.
In this section, assume $g:\bbT^2 \to \bbT^2$ is a weakly partially hyperbolic
diffeomorphism with a splitting of the form $T \bbT^2 = \Ec \oplus \Es$.
To be precise,
each of the one-dimensional subbundles $\Ec$ and $\Es$
is invariant under the derivative $Dg$ and
\[
    \|Dg v^s\| < \|Dg v^c\|
    \qandq
    \|Dg v^s\| < 1
      \]
hold for all $x  \in  \bbT^2$ and unit vectors
$v^s  \in  \Es(x)$ and 
$v^c  \in  \Ec(x)$.

Lift $g$ to a map on the universal cover.
By a slight abuse of notation,
we also denote the lifted map $\bbR^2  \to  \bbR^2$ by the letter $g$.
For the remainder of the section, we only
consider the lifted dynamics on $\bbR^2$.
This type of dynamical system is analyzed in detail in
\cite[Section 4.A]{potrie2012thesis}
and proofs of the next four propositions may be found there.

\begin{prop} \label{prop:hyptwo}
    There is a hyperbolic linear map $A:\bbR^2 \to \bbR^2$
    at finite distance from $g$.
\end{prop}

\begin{prop}
    There is a unique $\bbZ^2$-invariant, $g$-invariant foliation $\Wc_g$
    tangent to $\Ec_g$.
\end{prop}
Let $H:\bbR^2 \to \bbR^2$ be the Franks semiconjugacy \cite{Franks1}.
That is, $H$ is the unique continuous surjective map such that
$A H(x) = H g(x)$ and $H(x+z) = H(x) + z$
for all $x \in \bbR^2$ and $z \in \bbZ^2$.
This implies that
$H$ is a finite distance from the identity map.

\begin{prop} \label{prop:ctou}
    For $x,y \in \bbR^2$,
    $y \in \Wc_g(x)$ if and only if $H(y) \in \Wu_A(H(x))$.
\end{prop}

\begin{prop} \label{prop:twogps}
    For $x,y \in \bbR^2$,
    the curves $\Wc_g(x)$ and $\Ws_g(y)$ intersect in exactly one point.
\end{prop}
The next result concerns unique integrability of the center and
is proved in \cite{pujals2007integrability}.

\begin{prop} \label{prop:uniqueinttwo}
    Any curve tangent to $\Ec_g$ lies inside a leaf of $\Wc_g$.
\end{prop}
\begin{remark}
    The proof given in
    \cite{pujals2007integrability}
    has a small typo which could be a source of confusion to the reader.
    In the equation $f^{-n}(J_1) \subset W^s_K(f^{-n}(J_1))$
    near the end of section 4.2 of that paper,
    the $J_1$ on the left should actually be $J_2$.
\end{remark}
\medskip

We now state and prove several additional results which will be needed 
later in this paper.

\begin{prop} \label{prop:stos}
    For $x,y \in \bbR^2$,
    if $y \in \Ws_g(x)$, then $H(y) \in \Ws_A(H(x))$.
\end{prop}
\begin{proof}
    As $H$ is uniformly continuous,
    $d(g^n(x), g^n(y)) \to 0$
    implies that 
    \[    
        d(A^n H(x), A^n H(y)) = d(H g^n(x), H g^n(y)) \to 0.
    \]
    This occurs exactly when $H(y) \in \Ws_A(H(x))$.
\end{proof}
Note that \cref{prop:ctou} has an ``if and only if''
condition and \cref{prop:stos} does not.

Let $\piu,\pis: \bbR^2 \to \bbR$ be linear maps such that
$\ker \piu$ is the stable leaf of $A$ which passes through the origin
and $\ker \pis$ is the unstable leaf.
Define $\Hu = \piu \circ H$ and $\Hs = \pis \circ H$.

\begin{prop} \label{prop:shomeo}
    For any stable leaf $L$ of $g$,
    the restriction of $\Hs$ to $L$ is
    a homeomorphism from $L$ to $\bbR$.
\end{prop}
\begin{proof}
    Using propositions \ref{prop:ctou} and \ref{prop:twogps} for $x,y  \in  L$,
    \begin{align*}
        \Hs(y) = \Hs(x)  \quad \Leftrightarrow \quad 
        H(y)  \in  \Wu_A(H(x))  \quad \Leftrightarrow \quad 
        y  \in  \Wc_g(x)  \quad \Leftrightarrow \quad 
        y = x.
    \end{align*}
    This shows that $\Hs|_L$ is injective.
    If $L$ accumulated on a point $y  \in  \bbR^2$,
    then
    $\Wc_g(y)$ would intersect $L$ in multiple points,
    contradicting \cref{prop:twogps}.
    Hence,
    $L$ is properly embedded.
    If $\{x_n\}$ is a sequence in $L$ such that $\|x_n\| \to \infty$,
    then $\|H x_n\| \to \infty$.
    As \cref{prop:stos} implies that $\Hu x_n$ is constant,
    it must be that $|\Hs x_n| \to \infty$.
    From this, one may show that
    $\Hs|_L$ is surjective.
\end{proof}
In general,
the restriction of $\Hu$ to a center leaf will not be a homeomorphism.
However, it is still monotonic,
as we show after first establishing a few lemmas.

\begin{lemma}
    There is $R > 0$ such that for any $x  \in  \bbR^2$,
    the set $\bbR^2 \sans \Ws_g(x)$ has connected components
    $S^-_x$ and $S^+_x$ satisfying
    \[
        \{y  \in  \bbR^2 : \Hu(y) < \Hu(x) - R\} \subof S^-_x
        \qandq
        \{y  \in  \bbR^2 : \Hu(y) > \Hu(x) + R\} \subof S^+_x.
    \]  \end{lemma}
\begin{proof}
    By \cref{prop:stos} and the fact that $H$ is a finite distance
    from the identity,
    there is a uniform constant $R_0 > 0$ such that
    $|\piu(p) - \piu(q)| < R_0$
    for any two points $p$, $q$ on the same stable leaf of $g$.
    Hence,
    the components $S^-_x$ and $S^+_x$
    may be labelled so that
    \[
        \{y  \in  \bbR^2 : \piu(y) < \piu(x) - R_0\} \subof S^-_x
        \qandq
        \{y  \in  \bbR^2 : \piu(y) > \piu(x) + R_0\} \subof S^+_x.
    \]
    As $\Hu$ is a finite distance from $\piu$,
    the desired result holds.
\end{proof}
\begin{lemma} \label{lemma:halftohalf}
    If $k > 0$ is even,
    then
    $f^k(S^-_x) = S^-_{f^k(x)}$
    and
    $f^k(S^+_x) = S^+_{f^k(x)}$.
\end{lemma}
\begin{proof}
    Let $\lam$ denote the unstable eigenvalue of $A$.
    Then $A H = H g$ implies that
    $\lam \Hu = \Hu g$.
    As $\lam^k > 1$ and $H$ is surjective, there is
    a point $y  \in  \bbR^2$ such that both \linebreak
    $\Hu(y) > \Hu(x) + R$
    and
    $\lam^k \Hu(y) > \lam^k \Hu(x) + R$.
    Then %
    $g^k(S^+_x)$ and $S^+_{g^k(x)}$ intersect at
    $g^k(y)$ and the two sets are therefore equal.
\end{proof}
\begin{lemma} \label{lemma:premonotonic}
    If $y  \in  S^+_x$, then $\Hu(y)  \ge  \Hu(x)$.
\end{lemma}
\begin{proof}
    Suppose $y  \in  S^+_x$ and $\delta := \Hu(x) - \Hu(y) > 0$.
    Then $g^k(y)  \in  S^+_{g^k(x)}$ and
    \[
        \Hu(y) = \Hu(x) - \lam^k \delta > \Hu(x) - R
    \]
    for large positive even $k$.
    This gives a contradiction.
\end{proof}
\begin{prop} \label{prop:monotonic}
    Let $\alpha:\bbR \to \bbR^2$ be a parameterization of a center leaf of $g$.
    Then $\Hu \circ \alpha$ is monotonic.
    That is, up to possibly replacing $\alpha$ with the reverse
    parameterization,
    $\Hu \alpha(s)  \le  \Hu \alpha(t)$ holds
    for all $s  \le  t$.
\end{prop}
\begin{proof}
    This follows from \cref{prop:twogps} and \cref{lemma:premonotonic}.
\end{proof}
\begin{cor} \label{cor:twov}
    For each $v \in \bbR^2$,
    $H \inv(v)$ is
    either a point or a compact curve lying inside a center leaf.
\end{cor}
\begin{proof}
    This follows from
    propositions \ref{prop:ctou} and \ref{prop:monotonic},
    and the fact that $H$ is a finite distance from
    the identity.
\end{proof}

\section{Branching foliations} \label{sec:bran} %

We now list a number of properties which hold for all partially hyperbolic
diffeomorphisms in dimension 3.  These properties follow from
the branching foliation theory developed
by
Brin, Burago, and Ivanov \cite{BBI1, BI, BBI2}.

A
\emph{branching foliation}
on a Riemannian 3-manifold $M$
is a collection $\cF_0$ of immersed surfaces called \emph{leaves}
such that
\begin{enumerate}
    \item every leaf is complete under the Riemannian metric
    pulled back from $M$,

    \item no two leaves topologically cross,
    
    \item if a sequence of leaves converges in the compact-open topology,
    then the limit surface is also a leaf, and

    \item through every point of $M$ there is at least one leaf.
\end{enumerate}

\begin{thm} \label{thm:branching}
    Let $f$ be a partially hyperbolic diffeomorphism of a closed 3-manifold $M$
    such that
    $\Eu$, $\Ec$, and $\Es$ are oriented and $f$ preserves these orientations.
    Then there is a branching foliation $\Fcs_0$ on $M$
    such that each leaf is tangent to $\Ecs$.
\end{thm}
See \cite{BI} for further details and the proof of \cref{thm:branching}.
In the setting of the theorem,
let $\tM$ be the universal cover of $M$.
Lift the branching foliation $\Fcs_0$ on $M$ to a branching foliation $\Fcs$ on $\tM$ by
taking every possible lift of every leaf.
In this paper, we almost exclusively work with the lifted branching foliation
on the universal cover and
\cref{thm:branching} may be restated as follows.

\begin{cor} \label{cor:branching}
    Let $f$ be the lift of a partially hyperbolic diffeomorphism to the
    universal cover $\tM$.
    Then there is a branching foliation $\Fcs$ tangent to $\Ecs$ on $\tM$
    such that
    if a deck transformation $\gam$ : $\tM  \to  \tM$
    preserves the orientations of $\Eu$, $\Ec$, and $\Es$
    as bundles over $\tM$,
    then $\gam$ takes leaves to leaves.
    Similarly, if $f$ preserves these orientations, then $f$ maps leaves to
    leaves.
\end{cor}

The branching foliation on $\tM$ has the following properties,
as proved in \cite[Section 3]{BBI2}.

\begin{prop} \label{prop:ssat}
    Each leaf $L  \in  \Fcs$ is saturated by stable leaves.
    That is, if $x  \in  L$, then $\Ws(x)$ is a subset of $L$.
\end{prop}
\begin{prop} \label{prop:csdivide}
    Each leaf $L  \in  \Fcs$ is a properly embedded plane which separates $\tM$
    into two half spaces.  That is, $\tM \without L$ has two connected components
    $L^+$ and $L^-$.
\end{prop}
\begin{prop} \label{prop:uscunique}
    Each leaf of $\Fcs$ intersects an unstable leaf in at most one point.
\end{prop}
\begin{prop} \label{prop:volume}
    There is a uniform constant $C>0$ such that if $J$ is an unstable segment,
    then $\volume U_1(J)  \ge  C \cdot \length(J)$.
\end{prop}
Here $U_1(J)$ consists of all points in $\tM$ at distance less than 1 from $J$.

Applying \cref{thm:branching}
to $f \inv$, there is also a branching foliation
tangent to $\Ecu$.  Let $\Fcu$ denote the lifted branching foliation on $\tM$.

\begin{prop} \label{prop:clines}
    If $\Lcs$ is a leaf of $\Fcs$ and $\Lcu$ is a leaf of $\Fcu$, then
    any connected component of the intersection $\Lcs \cap \Lcu$
    is a topological line which is properly embedded in $\tM$.
\end{prop}
See \cite[Lemma 6.1]{HP2} for a proof of \cref{prop:clines}.

\begin{prop} \label{prop:cvolume}
    There is a uniform constant $C>0$ such that if $J$ is a center segment
    lying inside a leaf of $\Fcs$,
    then $\volume U_1(J)  \ge  C \cdot \length(J)$.
\end{prop}
\begin{proof}
    Suppose $J$ is a compact curve tangent to $\Ec$ lying in a leaf $L$ of $\Fcs$.
    Further suppose $J$ intersects a stable leaf in two distinct points.
    Then there is a circle which is the concatenation
    of a center segment and a stable segment, and 
    as $L$ is diffeomorphic to $\bbR^2$,
    the Jordon curve theorem implies that this circle bounds a disk, $D$.
    We may assume that $\Es$ is oriented so that everywhere along
    the boundary of $D$, $\Es$ is either tangent to the boundary or
    transerve to the boundary and pointing into $D$.
    Then,
    as in the proof of the Poincar\'e--Bendixson theorem,
    one can show that the flow along the stable direction either has a closed
    orbit or a fixed point in $D$.
    Since the stable foliation only consists of lines, 
    this gives a contradiction.

    Hence, $J \cap \Ws(x) = \{x\}$ for all $x  \in  J$.
    By \cref{prop:uscunique},
    \[
        J \cap \bigcup_{y  \in  \Ws(x)} \Wu(y) = \{x\}
    \]
    for all $x  \in  J$ as well.
    Using this,
    one may adapt the proof of \cref{prop:volume}
    given in \cite{BBI2}
    to apply to center segments.
\end{proof}
\section{Regions between tori} \label{sec:defOm} %

From now on,
assume that $f:M \to M$ is a partially hyperbolic
diffeomorphism on a closed oriented 3-manifold
and that there is at least one $cs$ or $cu$-torus.
We first consider dynamics on the closed manifold $M$,
but later in this section we lift to the universal cover.

\begin{prop} \label{prop:cudisjoint}
    No $cs$-torus intersects a $cu$-torus.
\end{prop}
\begin{proof}
    As $\Ecs$ and $\Ecu$ are transverse,
    such an intersection would consist of circles tangent to $\Ec$.
    A center circle is ruled out by \cref{prop:ctou}.
\end{proof}
\begin{prop} \label{prop:csdisjoint}
    No two distinct $cs$-tori intersect.
\end{prop}
Note that since $\Ecs$ is not uniquely integrable,
it is a priori possible for two $cs$-tori to intersect without coinciding.
This possibility is explored in \cite{ham20XXprox} where a proof
of \cref{prop:csdisjoint} is given.
The proof relies heavily on branching foliations and other tools specific
to dimension three, and is much more involved than the simple proof
of \cref{prop:cudisjoint} above.

Let $\cT \subset M$ be the union of all $cs$ and $cu$-tori.
By the above propositions,
$\cT$ consists of disjoint tori
and \cite{ham20XXprox} shows that
there are only finitely many.
Let $\{M_i\}$ be the collection of compact manifolds with boundary
obtained by cutting $M$ along $\cT$.
Since $f(\cT)=\cT$, there is an iterate of $f^k$ which maps each $M_i$
to itself and each torus in $\cT$ to itself.
For simplicity, we replace $f$ by an iterate and assume $k = 1$.
Throughout the proof of \cref{thm:main},
we freely replace $f$ by an iterate when convenient.

\begin{prop} \label{prop:Midiff}
    Each $M_i$ is diffeomorphic to $\bbT^2 \times [0,1]$.
\end{prop}
\begin{proof}
    This is basically a restatement of the main result of
    \cite{rhrhu2011tori}.
    Since $M$ supports a partially hyperbolic diffeomorphism,
    it is irreducible.
    Hence, $M_i$ is irreducible.
    Let $T$ be a boundary component of $M_i$.
    \Cref{prop:hyptwo} implies that $f|_T$ is homotopic
    to a hyperbolic toral automorphism
    and \cite[Theorem 2.2]{rhrhu2011tori} implies that $T$
    is incompressible.
    As such, $f|_{M_i}$ and $T$ satisfy the conditions of
    \cite[Theorem 1.2]{rhrhu2011tori},
    which implies that $M_i$ is diffeomorphic to $\bbT^2 \times [0,1]$.
\end{proof}

We have now established items (1)--(3) of \cref{thm:main}.  The rest of the paper
focuses on proving item (4).
To do this, we lift to the universal cover $\tM$.
Let $\Om \subof \tM$ be a closed 3-dimensional submanifold with boundary
such that each boundary component of $\Om$ quotients down to
a $cs$ or $cu$-torus in $M$ and such that
no surface which intersects the interior of $\Om$
quotients down to a $cs$ or $cu$-torus.
This submanifold may then be thought of as a covering space
for one of the $M_i$.
\Cref{prop:Midiff} implies that $\Om$ is diffeomorphic to $\bbR^2 \times I$
where $I \subset \bbR$ is a compact interval.

It will at times
be convenient to use coordinates on $\Om$ and discuss
linear maps from $\Om$ to $\bbR$.
Therefore, we simply assume that $\Om$ is equal to $\bbR^2 \times I$.
That is, we treat $\bbR^2 \times I$ as a subset of $\tM$
denoted by $\Om$.
The Riemannian metric on $\tM$ inherited from $M$
may differ from the standard Euclidean metric on $\bbR^2 \times I$.
However, distances and volumes measured
with respect to the two metrics differ by at most a constant factor.
Therefore, in our analysis, we freely assume that $\Om = \bbR^2 \times I$
is equipped with the Euclidean metric.

Since $\bbZ^2$ acts on $\Om$ via deck transformations, we adopt
the following notation:{}
if $p=(v,s) \in \bbR^2 \times I$ and $z \in \bbZ^2$,
then $p + z = (v,s) + z = (v+z, s)$.

As we are assuming $f:M \to M$ maps each $M_i$ to itself,
it follows that there is a lift of $f$ to the universal cover
which leaves $\Om$ invariant.
We also denote this lifted map $\tM \to \tM$ by the letter $f$.
Since $f|_\Om$ quotients down to a map on $M_i$,
there is a linear map
$A:\bbR^2 \to \bbR^2$ such that if $p \in \bbR^2 \times I$
and $z \in \bbZ^2$,
then $f(p + z)=f(p) + A z$.
By \cref{prop:hyptwo}, $A$ is hyperbolic.
This implies that there is a semiconjugacy between $f|_\Om$ and $A$
\cite{Franks1}.
We list several properties of this semiconjugacy.

\begin{prop} \label{prop:semiconj}
    There is a unique continuous surjective map
    $H:\bbR^2 \times I \to \bbR^2$
    and a constant $C > 0$
    such that if $p = (v,s) \in \bbR^2 \times I$ and $z \in \bbZ^2$
    then
    \begin{enumerate}
        \item $H f(p) = A H(p)$,

        \item $H(p+z) = H(p) + z$, and

        \item $\|H(p) - v\| < C$.
      \end{enumerate}  \end{prop}
As in \cref{sec:dimtwo}, 
let $\piu:\bbR^2 \to \bbR$ be a linear map such that $\ker \piu$
is the stable leaf of $A$ which passes through the origin.
Define $\Hu = \piu \circ H$.

Now consider a branching foliation $\Fcs$ on $\tM$ as in \cref{cor:branching}.
We only know a priori that $\Fcs$ is invariant under those
deck transformations which preserve the orientations of the subbundles.
In particular,
for a full-rank subgroup $Z_0 \subof \bbZ^2$ it holds that
for any $z  \in  Z_0$, there is a deck transformation $\gam_z$ : $\tM  \to  \tM$
which preserves the orientations of the subbundles and
such that
\begin{math}
    \gam_z(p) = p+z
\end{math}
for all $p  \in  \bbR^2 \ti I = \Om$.
Therefore if $L$ is a leaf of $\Fcs$, then $\gam_z(L)$ is a leaf of $\Fcs$
as well.
Replacing $f$ by an iterate,
we may freely assume that
$f$ preserves the orientations of the subbundles on $\tM$.
Then the branching foliation $\Fcs$ is invariant under $f$.

Let $\Omin$ denote the interior of $\Om$.
A major step is to relate the branching foliation $\Fcs$
to the semiconjugacy $H$ for points in $\Omin$.

\begin{prop} \label{prop:csH}
    For $p,q \in \Omin$,
    $\Hu(p) = \Hu(q)$
    if and only if
    there is $L \in \Fcs$ such that $p,q \in L$.
\end{prop}
This is proved first in \cref{sec:csfiber} for the specific case
where at least one component of $\del \Om$ is tangent to $\Ecu$.
\Cref{sec:hiddentorus} proves the result in the case
where both components of $\del \Om$ are tangent to $\Ecs$.
The proof of the latter case relies on the proof of the former,
and this significantly complicates the exposition.
However, we know of no simpler method to establish \cref{prop:csH}.

An analogous statement also holds for $\Hs = \pis \circ H$
and any branching foliation $\Fcu$ tangent to $\Ecu$.
    
\begin{cor} \label{cor:cuH}
    For $p,q \in \Omin$,
    $\Hs(p) = \Hs(q)$
    if and only if
    there is $L \in \Fcu$ such that $p,q \in L$.
\end{cor}
After these results are established,
they are used in
\cref{sec:fibers} 
to prove the following.

\begin{prop} \label{prop:orient}
    If $\gam$ : $\tM  \to  \tM$ is a deck transformation such that
    $\gam(\Om) = \Om$,
    then $\gam$ preserves the orientations of $\Eu$, $\Ec$, and $\Es$
    as bundles over $\tM$.
\end{prop}
This shows that $Z_0$ above may be taken as equal to $\bbZ^2$.
\Cref{sec:fibers} also proves the following
characterization of the fibers of the semiconjugacy.

\begin{prop} \label{prop:fiber}
    For every $v \in \bbR^2$, the pre-image
    $H \inv(v)$ is a compact segment tangent to $\Ec$.
    Moreover, $H \inv(v)$ intersects each boundary component of $\Om$
    in either a point or a compact segment.
\end{prop}
\Cref{sec:leafconj} uses this to construct the topological conjugacy
given in \cref{thm:main}.

\section{Center-stable leaves} \label{sec:csfiber} %

This section gives the proof of \cref{prop:csH}
under certain assumptions.  These assumptions are removed in the next section.
Let $f$, $\Om$, $H$, and $\Fcs$
be as in the previous section.
By abuse of notation,
we consider $\piu$ and $\pis$ as both linear maps from $\bbR^2$ to $\bbR$
and as maps from $\Om = \bbR^2 \times I$ to $\bbR$ which depend only on the
$\bbR^2$ coordinate.

\begin{prop} \label{prop:segubound}
    For any constant $D>0$, there is $\ell > 0$ such that
    any unstable curve $J \subset \Om$ of length at least $\ell$
    contains points $p$ and $q$ with
    $|\pi^u(p) - \pi^u(q)| > D$.
\end{prop}
\begin{proof}
    This result is analogous to \cite[Lemma 4.7]{ham-nil}
    and the proof given there applies here
    with only minor modifications.
\end{proof}

Define $\delcs$ as the union of those components of $\del \Om$ which are tangent
to $\Ecs$.
Let $d_u$ be distance measured along an unstable leaf and
define $\Ku$ as the largest subset of $\Om$ for which the following
property holds{:}
if $p \in \Ku$, $q \in \Wu(p)$, and $d_u(p,q) < 1$, then $q \notin \delcs$.
In other words, $\Ku$ is the set of points at distance at least 1
from $\delcs$ where the distance
is measured along the unstable direction.

Note that there are three possibilities.  If both components of $\del \Om$
are tangent to $\Ecs$, then $\Eu$ is transverse to $\del \Om$
and $\Ku$ lies in the interior $\Omin$ of $\Om$.
If both components of $\del \Om$ are tangent to $\Ecu$,
then $\delcs$ is empty and $\Ku = \Om$.
Finally, if one component of $\del \Om$ is tangent to $\Ecu$ and the other
is tangent to $\Ecs$, then $\Ku$ contains one boundary component
and not the other.
Fortunately,
most of the arguments in the remainder of the paper
do not depend on which case we are in.

One can verify that $\Ku$ is a closed subset of $\Om$
and is invariant under any deck transformation which fixes $\Om$.
As $f$ increases distances measured along the unstable direction,
it follows that $f(\Ku) \subset \Ku$.
Note that if $J$ is a compact subset of $\Omin$,
then there is an integer $N(J)$ such that $f^n(J) \subset \Ku$
for all $n > N(J)$.

We now consider the intersection of $\Ku$ 
with the leaves of the branching foliation $\Fcs$.

\begin{prop} \label{prop:linfoln}
    There is a non-zero map $\pi:\Om \to \bbR$
    of the form $\pi = a \piu + b \pis$ with constants $a,b \in \bbR$
    such that
    if $L \in \Fcs$ and $p,q \in \Ku \cap L$ then
    $|\pi(p) - \pi(q)| < 1$.
\end{prop}
We prove this
by adapting techniques presented in \cite{BBI2,HP1}.
The proof is largely topological in nature, instead of involving
the dynamics acting on $\Om$.
Therefore, we defer the proof of \cref{prop:linfoln} to the appendix.

\begin{lemma} \label{lemma:stabonce}
    No stable or unstable leaf intersects both boundary components
    of $\Om$.
\end{lemma}
\begin{proof}
    Note that there is a uniform lower bound
    on the distance between points in the two boundary components of $\Om$.
    If a stable or unstable segment $J$
    had endpoints on both boundary components,
    one could find $n  \in  \bbZ$ such that the length of $f^n(J)$ was smaller than
    this lower bound, and this would give a contradiction.
\end{proof}

\begin{lemma} \label{lemma:sgrow1}
    If $\delcs  \ne  \del \Om$, then
    for any $c > 0$,
    there is $L  \in  \Fcs$ and $p$, $q  \in  \Ku \cap L$ such that
    $|\pis(p) - \pis(q)| > c$.
\end{lemma}
\begin{proof}
    If $\delcs  \ne  \del \Om$, then
    $\Om$ has a boundary component $S$ tangent to $\Ecu$.
    By \cref{lemma:stabonce}, $S$ is contained in $\Ku$.
    For any $x  \in  S$,
    there is a leaf $L  \in  \Fcs$ through $x$.
    By \cref{prop:uniqueinttwo} where $g$ is given by the restriction
    of $f \inv$ to $S$,
    the intersection of $S$ and $L$ contains $\Wc_g(x)$.
    The result then follows by \cref{prop:ctou}
\end{proof}
Assume for the remainder of the section that
the conclusion of \cref{lemma:sgrow1} holds.

\begin{prop} \label{prop:csbound}
    There is $R>0$ such that if
    $L \in \Fcs$ and $p,q \in \Ku \cap L$ then\\
    $|\pi^u(p) - \pi^u(q)| < R$.
\end{prop}
\begin{proof}
    Define a linear map $\piA$ : $\Om  \to  \bbR$ by
    \[
        \piA(v,s) = \pi(A(v),s)
      \]
    for $(v,s)  \in  \Om = \bbR^2 \ti I$
    where $\pi$ is given by \cref{prop:linfoln}.

    If $L  \in  \Fcs$ and $p,q  \in  \Ku \cap L$,
    then
    $f(\Ku) \subof \Ku$ implies that
    $f(p),f(q)  \in  \Ku \cap f(L)$
    and therefore
    \begin{math}
        |\pi f(p) - \pi f(q)| < 1.
    \end{math}
    Since $f$ is a finite distance from $A \ti \id$ on $\Om$,
    there is $C > 0$ such that
    \begin{math}
        |\piA(p) - \piA(q)| < C
    \end{math}
    for all such $p$ and $q$.
    If both constants $a$ and $b$ are non-zero in \cref{prop:linfoln},
    then $\pi$ and $\piA$ are linearly independent
    and $\piu$ is a linear combination of $\pi$ and $\piA$.
    From this, the result would follow and
    therefore we may assume that one of $a$ or $b$ is zero.
    The conclusion of \cref{lemma:sgrow1}
    implies that the latter case must hold.
\end{proof}

For a point $p \in \Om$, define
\[
    K^-(p) = \{ q \in \Ku : \pi^s(q)  \le  \pi^s(p) - R \}
\]
and
\[
    K^+(p) = \{ q \in \Ku : \pi^s(q)  \ge  \pi^s(p) + R \}.
\]
Replacing $f$ by $f^2$ if necessary,
we assume that $A$ has positive eigenvalues.
The fact that $f$ is at finite distance from $A \times \id$
then implies that $K^+(f^n(p))$ intersects $f^n(K^+(p))$
for all $n$.

\begin{prop} \label{prop:leafsplit}
    If $L \in \Fcs$ and $p \in \Ku \cap L$, then $\tM \without L$
    has connected components $L^-$ and $L^+$ such that
    $K^-(p) \subset L^-$ and
    $K^+(p) \subset L^+$.
\end{prop}
\begin{proof}
    \Cref{prop:csbound} shows that
    $L$ is disjoint from both $K^-(p)$ and $K^+(p)$.
    Therefore, it is enough to show that each of $L^-$ and $L^+$
    intersects at least one of $K^-(p)$ or $K^+(p)$.

    Suppose instead that $K^-(p) \cup K^+(p) \subset L^-$.
    Then for any $n  \ge  0$,
    $K^+(f^n(p))$ intersects $f^n(K^+(p))$ and
    is therefore a subset of $f^n(L^-)$.
    Similarly for $K^-(f^n(p))$.
    Since $p \in \Ku \cap L$, the open set $\Omin \cap L^+$ is non-empty.
    Let $J$ be a small unstable segment lying in $\Omin \cap L^+$.
    Then $f^n(J) \subset \Ku$ for all large $n$.
    Since $f^n(J) \cap f^n(L^-)$ is empty,
    the length of $\pi^u f^n(J)$ is bounded by $2R$
    for all large $n$.
    However, \cref{prop:segubound} shows that there is no uniform bound
    on the length of $\pi^u f^n(J)$
    and gives a contradiction.
\end{proof}
\begin{prop} \label{prop:csbdd}
    For $p,q \in \Ku$, the following are equivalent{:}
    \begin{itemize}
        \item $\sup_{n  \ge  0} |\pi^u f^n(p) - \pi^u f^n(q)| < \infty$, and

        \item there is $L \in \Fcs$ such that $p,q \in L$.
    \end{itemize}  \end{prop}
\begin{proof}
    One direction follows from \cref{prop:csbound}
    and the fact that $f(\Ku) \subof \Ku$.
    To prove the other direction, suppose $p \in L_p \in \Fcs$
    and $q \in L_q \in \Fcs$.
    Let $L_p^-$ and $L_p^+$ be as in the previous proposition
    and let
    $L_q^-$ and $L_q^+$ be the corresponding sets associated to $L_q$.
    Assume $q$ does not lie on $L_p$.
    Then $q$ lies either in $L^-_p$ or $L^+_p$.
    Without loss of generality,
    assume $q \in L^-_p$.
    Since $L_q$ is the boundary of $L_q^+$,
    it follows that 
    $L_p^- \cap L_q^+ \cap \Omin$ is a non-empty open set.
    Consider a small unstable curve
    $J$ in this set.
    Then $f^n(J) \subset \Ku$ and
    \[
        f^n(J) \cap
        \bigl( K^+(f^n(p)) \cup K^-(f^n(q)) \bigr)
        = \varnothing
    \]
    for all large $n$.
    The assumption that
    $|\pi^u f^n(p) - \pi^u f^n(q)|$
    is uniformly bounded
    implies that the length of
    $\pi^u f^n(J)$ is uniformly bounded.
    \Cref{prop:segubound} again gives a contradiction.
      \end{proof}
\begin{proof}
    [Proof of \cref{prop:csH}]
    By the properties of the semiconjugacy,
    \begin{align*}
        \Hu(p) = \Hu(q)& \quad \Leftrightarrow \quad 
        \sup_{n \ge 0}  |\piu A^n H(p) - \piu A^n H(q)| < \infty \\
        & \quad \Leftrightarrow \quad 
        \sup_{n \ge 0}  |\piu f^n(p) - \piu f^n(q)| < \infty.
    \end{align*}
    Since $f^n(p)$ and $f^n(q)$ are in $K^u$ for all large $n$,
    the result follows from \cref{prop:csbdd}.
\end{proof}
This proof was conditional on the conclusion of \cref{lemma:sgrow1}
and therefore we have only established \cref{prop:csH}
in the case where $\delcs  \ne  \del \Om$.
The next section gives a replacement for \cref{lemma:sgrow1}
in the case where $\delcs = \del \Om$ and will therefore finish the proof
of \cref{prop:csH}.

\section{Finding a hidden torus} \label{sec:hiddentorus} %

The goal of this section is to prove the following.

\begin{lemma} \label{lemma:sgrow2}
    If $\delcs = \del \Om$, then
    for any $c > 0$,
    there is $L  \in  \Fcs$ and $p$, $q  \in  \Ku \cap L$ such that
    $|\pis(p) - \pis(q)| > c$.
\end{lemma}
Note that up to replacing $f$ with $f \inv$, \cref{lemma:sgrow2} is equivalent
to the following.

\begin{lemma} \label{lemma:sgrowinv}
    If $\delcs = \varnothing$, then
    for any $c > 0$,
    there is $L  \in  \Fcu$ and $x$, $y  \in  \Ks \cap L$ such that
    $|\piu(x) - \piu(y)| > c$.
\end{lemma}
Here, $\Ks$ is the set of points at distance at least 1
from $\del \Om$ where distance is measured along the stable direction.
The advantage in proving \cref{lemma:sgrowinv} in place of \cref{lemma:sgrow2} is that,
since $\delcs = \varnothing$, all of the results of the previous section
are known to hold for $\Fcs$ and we may use those properties of $\Fcs$
when proving results for $\Fcu$.

We prove \cref{lemma:sgrowinv} by contradiction.
Therefore,
for the remainder of the section,
assume that $\del \Om$ is tangent to $\Ecu$,
that $\Fcu$ is a branching foliation tangent to $\Ecu$,
and that there is $D_0 > 0$ such that
if $L  \in  \Fcu$ and $x$, $y  \in  \Ks \cap L$, then
$|\piu(x) - \piu(y)| < D_0$.
We will use a result in \cite{ham20XXprox}
to establish the existence of a surface lying in $\Omin$
which quotients down to a $cs$-torus in the original
compact 3-manifold $M$.
Since $\Om$ was chosen so that no such torus exists,
this will provide the needed contradiction.

\begin{lemma}
    There is $D > 0$ such that
    if $L  \in  \Fcu$ and $x$, $y  \in  \Ks \cap L$, then
    $|\Hu(x) - \Hu(y)| < D$.
\end{lemma}
\begin{proof}
    This follows from the inequality with $D_0$ above
    and the fact that $\Hu$ and $\piu$ are at finite distance.
\end{proof}
\begin{lemma} \label{lemma:xyD}
    There is $\ell > 0$ such that
    if $x  \in  \Omin$, $y  \in  \Wu(x)$, and $d_u(x,y) > \ell$,
    then $|\Hu(x) - \Hu(y)| > D$.
\end{lemma}
\begin{proof}
    Let $C > 1$ be such that
    for any $x  \in  \Om$ and $y  \in  \Wu(x)$,
    there is an integer $k$ such that
    $1  \le  d_u(f^k(x),f^k(y))  \le  C$.
    Define
    \[
        X =
        \{ (x,y) : x  \in  \Om, \, y  \in  \Wu(x), \text{ and } 1  \le  d_u(x,y)  \le  C\}.
    \]
    Under the assumptions of the current section,
    $\delcs = \varnothing$
    and therefore the last section shows that the conclusions of
    \cref{prop:csH} hold for $\Fcs$.
    As such, \cref{prop:uscunique} implies that
    $\Hu(x) - \Hu(y)$ is non-zero for all $(x,y)  \in  X$.
    As $X$ may be quotiented down to a compact set,
    there is $\delta > 0$ such that
    $|\Hu(x) - \Hu(y)| > \delta$ for all $(x,y)  \in  X$.

    The semiconjugacy relation $H f = A H$
    implies that
    $\Hu(f(x)) = \lam \Hu(x)$
    where $\lam > 1$ is the unstable eigenvalue of $A$.
    Choose $n$ such that $\lam^n \delta > D$.
    Then
    $|\Hu(x) - \Hu(y)| > \delta$
    implies
    $|\Hu f^n(x) - \Hu f^n(y)| > \lam^n \delta > D$.
    To conclude the proof, choose $\ell > 0$
    so that
    $d_u(x,y) > \ell$ implies
    $d_u(f^{-n}(x), f^{-n}(y)) > 1$.
\end{proof}
\begin{cor}
    If $x  \in  \Ks$, $y  \in  \Wu(x)$, and $d_u(x,y) > \ell$,
    then $y \notin \Ks$.
\end{cor}
\begin{proof}
    Otherwise, the results above give
    $D < |\Hu(x) - \Hu(y)| < D$.
\end{proof}
\begin{lemma}
    There is a continuous function $g$ : $\Om  \to  [0,1]$ such that
    \begin{enumerate}
        \item $g$ is invariant under deck transformations;

        \item $g(\del \Om) = \{0,1\}$; and

        \item if $0 < g(x) < 1$, then $x  \in  \Ks$.
    \end{enumerate}  \end{lemma}
\begin{proof}
    Let $S_0$ and $S_1$ be the two boundary components of $\Om$.
    Define
    \[
        K_i =
        \{ x  \in  \Ws(y) \cap \Om : y  \in  S_i, \,  d_s(x,y)  \le  1 \}.
    \]
    By \cref{lemma:stabonce},
    no stable manifold intersects both $S_0$ and $S_1$
    and therefore $K_0$ and $K_1$ are disjoint.
    Define
    \[
        g(x) = 
        \frac{\dist(x,K_0)}{\dist(x,K_0)+\dist(x,K_1)}.\qedhere
    \]  \end{proof}
Now $\Om$ quotients down to a subset $M_0 \subof M$
and $g$ quotients down to a function from $M_0$ to [0,1].
Applying \cite[Theorem 2.5]{ham20XXprox},
there is a compact $cs$-submanifold in the interior of $M_0$.
This contradicts the assumptions on $\Om$ given in \cref{sec:defOm}
and completes the proof of \cref{lemma:sgrowinv}.
Since the two statements are equivalent,
this also proves \cref{lemma:sgrow2}.
Now,
substituting \cref{lemma:sgrow2} in place of \cref{lemma:sgrow1}
in the previous section, one sees that \cref{prop:csH}
holds in full generality.

\section{Fibers of the semiconjugacy} \label{sec:fibers} %

This section gives the proofs of propositions \ref{prop:orient}
and \ref{prop:fiber}.
Let $f$, $\Om$, and $H$
be as in \cref{sec:defOm}.
Recall that $\delcs$ is the union of those boundary components of $\Om$
which are tangent to $\Ecs$.

\begin{lemma} \label{lemma:notwice}
    An unstable curve intersects
    $\delcs$ in at most one point.
\end{lemma}
\begin{proof}
    Suppose $J$ is an unstable segment which intersects
    $\delcs$ at both endpoints.
    By \cref{lemma:stabonce},
    both endpoints must lie on the same boundary component.
    Then for large $n$, $f^{-n}(J)$ would be an arbitrarily small unstable
    curve connecting two points on the same center-stable surface.
    This is ruled out by the
    uniform transversality of $\Ecs$ and $\Eu$.
\end{proof}
\begin{lemma} \label{lemma:uhomeo}
    If $J$ is an unstable curve in $\Om$,
    then
    $\Hu|_J$ is a homeomorphism to its image.
\end{lemma}
\noindent
Here, the curve $J$ may be bounded or unbounded
and may or may not include its endpoints.

\begin{proof}
    First,
    consider the case where $J$ is in the interior of $\Om$.
    By \cref{prop:csH}, the fibers of $\Hu$ are $cs$-leaves and,
    by \cref{prop:uscunique}, each
    $cs$-leaf intersects an unstable leaf at most once,
    so $\Hu|_J$ is injective.
    Since $\Hu$ is continuous,
    this implies that $\Hu|_J$ is a homeomorphism to its image.

    Note that if $\phi:[0,1) \to \bbR$ is a continuous function
    and $\phi|_{(0,1)}$ is an embedding, then $\phi$ itself must also be an
    embedding. Therefore, in the case where $J$ intersects $\del \Om$
    in a point, the fact that the restriction of $\Hu$ to
    $J \without \del \Om$ is injective
    implies that $\Hu$ is injective on all of $J$.

    The last possibility is if $J$ lies in a component of $\del \Om$
    which is tangent to $\Ecu$.
    This case follows from \cref{prop:shomeo}.
\end{proof}
\begin{proof}
    [Proof of \cref{prop:orient}]
    For a point $x  \in  \Omin$,
    let $J \subof \Omin$ be a short unstable segment
    passing through $x$.
    Using that $\Hu|_J$ is injective,
    define the orientation for $\Eu$ near $x$ so that
    $\Hu$ increases along $J$.
    This gives a well-defined continuous orientation
    of $\Eu$ on all of $\Omin$.

    Suppose $\gam$ : $\tM  \to  \tM$ is a deck transformation
    mapping $\Om = \bbR^2 \ti I$ to itself.
    By the properties of the semiconjugacy,
    there is $z  \in  \bbZ^2$ such that
    $H \gam(x) = H(x) + z$
    for all $x  \in  \Omin$.
    Hence, $\gam$ preserves the orientation of $\Eu$.
    An analogous argument shows that
    $\gam$ preserves the orientation of $\Es$.
    By assumption,
    the original closed 3-manifold $M$ is orientable.
    Therefore,
    $\gam$ preserves the orientation of $T \tM$
    and must also
    preserve the orientation of $\Ec$.
\end{proof}
\begin{lemma} \label{lemma:Hunbdd}
    For $x \in \delcs$,
    the set $\Hu(\Wu(x) \cap \Om)$ is equal
    either to  $(-\infty, \Hu(x)]$ or $[\Hu(x), +\infty)$.
\end{lemma}
\begin{proof}
    For $x \in \delcs$, let $J_1(x)$ be the compact unstable segment
    which starts at $x$, is directed into $\Om$, and
    has length exactly one.
    By \cref{lemma:uhomeo}, $\Hu(J_1(x))$ always has positive length.
    By a compactness argument, there is $\delta > 0$
    such that the length of $\Hu(J_1(x))$ is greater than $\delta$
    for all $x$.
    As in the proof of \cref{lemma:xyD},
    one may show that
    \[
        \length \Hu f^n J_1(x) > \lam^n \delta
    \]
    for all $x \in \delcs$ and $n  \ge  0$
    where $\lam > 1$ is the unstable eigenvalue of $A$.
    \Cref{lemma:notwice}
    implies that $f^n J_1(f^{-n}(x))$ is a subset of $\Wu(x) \cap \Om$
    for all $n$.
    Thus, the length of $\Hu(\Wu(x) \cap \Om)$ is unbounded.
\end{proof}
\begin{lemma} \label{lemma:fiberU}
    Let $S$ be a connected component of $\delcs$
    and let
    \[    
        U = \Om \cap \bigcup_{x  \in  S} \Wu(x).
    \]
    For any $v \in \bbR^2$, $H \inv(v) \cap U$ is a topological ray.
    That is, there is a proper topological
    embedding $\beta:[0,+\infty) \to U$
    such that the image is $H \inv(v) \cap U$.
\end{lemma}
\begin{proof}
    First, note that $H \inv(v) = (\Hu) \inv (q) \cap (\Hs) \inv (r)$
    for some pair of numbers $q,r \in \bbR$.
    By \cref{cor:cuH},
    there is a leaf $\Lcu \in \Fcu$ such that
    $(\Hs) \inv (r) \cap U = \Lcu \cap U$.
    Therefore, we may restrict our attention to this leaf.

    Consider the intersection of $\Lcu$ and $S$.
    The semiconjugacy $H$ : $\Om  \to  \bbR^2$
    when restricted to $S$ agrees with the semiconjugacy,
    also denoted $H$,
    studied in \cref{sec:dimtwo}.
    In particular,
    $H|_S$ is surjective.
    This implies that $\Lcu \cap S$
    is non-empty.
    Every connected component of $\Lcu \cap S$ is a center line.
    By \cref{prop:twogps},
    any stable leaf lying in $S$ intersects every connected component
    of $\Lcu \cap S$.
    By \cref{prop:uscunique},
    $\Lcu$ intersects a stable curve in at most one point and so
    $\Lcu \cap S$ has exactly one connected component.
    Thus, $\Lcu \cap S$ is a properly embedded center curve
    and may be parameterized by a function $\alpha:\bbR \to \Lcu \cap S$.

    As $S$ is connected, %
    exactly one of the two cases in
    \cref{lemma:Hunbdd} holds for all $x  \in  S$.
    \WLOG{}, assume the case
    $\Hu(\Wu(x) \cap \Om) = (-\infty, \Hu(x)]$ holds.
    By \cref{prop:monotonic},
    $\alpha$ is monotonic.
    Up to composing $\alpha$ by an affine map on $\bbR$,
    we may assume $\alpha$ is defined so that $\Hu \alpha(t)  \ge  q$
    exactly when $t  \ge  0$.
    Define a map $\beta:[0,+\infty) \to U$
    by setting $\beta(t)$ to be the unique point in
    $\Wu (\alpha(t))$ which satisfies $\Hu \beta(t) = q$.
    Proving that $\beta$ is continuous
    reduces to proving the following claim.
    \begin{quote}
        \textbf{Claim.} Suppose
        $h:[0, +\infty) \times [0, +\infty) \to \bbR$
        is a continuous function with the properties that
        $x_1  \le  x_2$ implies $h(x_1,0)  \le  h(x_2,0)$
        and $y_1 < y_2$ implies $h(x,y_2) < h(x,y_1)$.
        Then, any level set of $h$ is the graph of a continuous function.
    \end{quote}
    The proof of the claim is left to the reader.
    In fact, the proof is highly similar in form to steps used in proving
    the implicit function theorem.

    It is clear that $\beta$ is injective.
    Suppose a sequence $\{t_k\}$ tending to $+\infty$ is such that
    $\beta(t_k)$ converges to a point $x \in U$.
    Since $\Wu(x)$ intersects $S$,
    one may use an unstable foliation chart in a neighbourhood of $x$,
    to derive a contradiction.
    This shows that $\beta$ is proper.
\end{proof}

\begin{proof}
    [Proof of \cref{prop:fiber}]
    Consider $v  \in  \bbR^2$.
    By \cref{lemma:fiberU},
    there is at least one point $x \in \Omin$ such that $H(x)=v$.
    Let $\Lcs$ be a leaf of $\Fcs$ passing through $x$
    and $\Lcu$ a leaf of $\Fcu$.
    Let $L$ be the connected component of $\Lcs \cap \Lcu$ which 
    passes through $x$.
    By \cref{prop:clines},
    $L$ is a properly embedded line.
    Since $H$ is proper,
    $H \inv(v)$ is a compact subset of $\Om$, and
    the ends of $L$ must eventually leave the interior of $\Om$.
    As such, there is a compact center segment $J \subset L$
    such that the endpoints of $J$ lie on $\del \Om$ and all other points
    of $J$ lie in the interior of $\Om$.
    \Cref{lemma:fiberU} then describes the exact shape of $J$ near the boundary
    of $\Om$.
    In particular, one sees that the two endpoints of $J$ cannot lie on the
    same boundary component, and so each of the two
    boundary components contains
    exactly one endpoint of $J$.

    Now suppose another connected component of $\Lcs \cap \Lcu$
    intersected the interior of $\Om$.
    This would lead to a center segment $J'$ disjoint from $J$ but such
    that $H(J) = H(J')$ and where each boundary component of $\Om$ contained
    exactly one endpoint of $J'$.
    \Cref{lemma:fiberU} would then imply that $J$ and $J'$ coincide near the
    boundary of $\Om$, a contradiction.
    Thus $H \inv(v)$ consists of $J$ together with the pre-images of $v$ on the
    two boundary components of $\Om$.  By \cref{cor:twov} and
    \cref{lemma:fiberU}, the result follows.
      \end{proof}
\begin{cor} \label{cor:fiberlength}
    There is a uniform upper bound on the length of a fiber $H$ $\inv(v)$.
\end{cor}
\begin{proof}
    As $H$ is proper and commutes with deck transformations,
    there is a uniform upper bound on the diameter $H$ $\inv(v)$
    independent of $v  \in  \bbR^2$.
    There is then a uniform upper bound on the volume of
    $U_1 (H \inv(v))$ and the result follows from \cref{prop:cvolume}.
\end{proof}

\section{Building the ragged leaf conjugacy} \label{sec:leafconj} %

\begin{lemma} \label{lemma:fuller}
    There is a continuous function $p: \Omin \to (0,1)$
    such that
    for any center segment
    of the form $J = H \inv(v) \cap \Omin$,
    the restriction $p|_{J}$ is a $C^1$ embedding.
    Moreover, with respect to arc length, $p|_{J}$ has a uniform speed
    independent of the choice of $J$.
\end{lemma}
\begin{proof}
    Let $S_0$ and $S_1$ be the two boundary components of $\Om$.
    For $x \in \Omin$, define
    \[    
        p_0(x) =
        \frac
            {\dist(x, S_0)}
            {\dist(x, S_0) + \dist(x, S_1)}.
    \]
    Extend $p_0$ to a continuous function $p_0 : \tM \to [0,1]$
    by requiring it to be locally constant outside of $\Om$.
    
    Now suppose $x \in \Omin$ and
    let $\Lcs \in \Fcs$ and
    $\Lcu \in \Fcu$ be leaves of the foliations
    such that
    $x \in \Lcs \cap \Lcu$.
    Let $\alpha:\bbR \to \Lcs \cap \Lcu$
    be a parameterization by arc length of this center curve.
    By \cref{prop:clines},
    $\alpha(\bbR)$ is a complete curve properly embedded in $\tM$.
    Let $T > 0$ be the upper bound given by
    \cref{cor:fiberlength} and
    for any $t \in \bbR$, define
    \[
        p( \alpha(t) ) = \frac{1}{2T} \int_{t-T}^{t+T} p_0( \alpha(s) ) \, ds.
    \]
    If $\alpha(t) \in \Omin$, then neither $\alpha(t-T)$ nor $\alpha(t+T)$
    lies in $\Omin$.
    Up to possibly reversing the parameterization,
    it follows that $\alpha(t-T) = 0$ and $\alpha(t+T) = 1$
    and by the Fundamental Theorem of Calculus
    \[
        \frac{d}{dt} p ( \alpha(t) ) = \frac{1}{2T}.\qedhere
    \]  \end{proof}
\begin{lemma} \label{lemma:xyz}
    If $x,y,z  \in  \Omin$ with $y  \in  \Wu(x)$ and $z  \in  \Ws(y)$, then
    $H(x) = H(z)$ if and only if $x = y = z$.
\end{lemma}
\begin{proof}
    Assume $H(x) = H(z)$.
    By \cref{prop:csH},
    $x$ and $z$ lie on the same leaf $L$ of $\Fcs$.
    By \cref{prop:ssat},
    $L$ is saturated by stable curves and therefore
    the point $y$ also lies in $L$.
    The uniqueness given by 
    \cref{prop:uscunique} implies that $x = y$.
    A similar argument shows that $y = z$.
\end{proof}
\begin{proof}
    [Proof of \cref{thm:main}]
    Define $h = H \times p : \Omin \to \bbR^3$.
    This is a continuous map and injective by \cref{prop:fiber} and
    \cref{lemma:fuller}.
    We now show that $h$ is an open map.
    Consider $x  \in  \Omin$ and
    assume
    that $\Eu$, $\Ec$, and $\Es$ are oriented.
    With respect to these orientations,
    define unit speed flows $\varphi^s$ and $\varphi^u$
    along the stable and unstable directions.
    Define $\varphi^c$ as the unit speed flow along the fibers of $H$
    with the direction given by the orientation of $\Ec$.
    Define $i$ : $[-\ep,\ep]^3  \to  \Omin$ by
    \[
        i(t_1,t_2,t_3) = \varphi^c_{t_3} \varphi^s_{t_2} \varphi^u_{t_1} (x)
    \]
    where $\ep$ is small enough that the range of $i$ is contained in
    $\Omin$.
    By lemmas \ref{lemma:fuller} and \ref{lemma:xyz},
    the range of $h \circ i$ contains $h(x)$ in its interior.
    Taking $\ep$ to zero,
    one can show that for any neighbourhood $V$ of $x$,
    $h(V)$ is a neighbourhood of $h(x)$
    and therefore $h$ is open.
    It follows that $h$ is
    a homeomorphism to its image, $U := h(\Omin)$.

    As $H$ is a semiconjugacy,
    the homeomorphism $h f h \inv:U \to U$ is of the form
    \[
        h f h \inv(v, s) = (A(v), \phi(v,s) )
    \]
    where $A:\bbR^2 \to \bbR^2$ is the hyperbolic linear map and
    $\phi:U \to \bbR$ is some continuous function.
    By construction, the maps $H$, $p$, and therefore $h$
    are $\bbZ^2$-equivariant.
    Thus, $h$ quotients down to an embedding
    $U_i  \to  \bbT^2 \ti \bbR$
    which satisfies the properties listed in item (4) of \cref{thm:main}.
\end{proof}

\appendix \section{Bounds on foliations} \label{sec:ap} %

This appendix proves \cref{prop:linfoln}.
Let $\Fcs$, $\Om = \bbR^2 \ti I$, and $Z_0 \subof \bbZ^2$
be as in \cref{sec:defOm}.
Let $S$ be a leaf of $\Fcs$ which intersects the interior $\Omin$ of $\Om$.
Here, we use $S$ instead of $L$ to keep closer to the notation of \cite{BBI2}.
As $S$ is properly embedded,
the complement $\tM \without S$ consists of two open
connected components $S_+$ and $S_-$
where the oriented $\Eu$ bundle points into $S_+$.
Define $S_+ + Z_0 = \{p + z : p \in S_+, z \in Z_0\}$
and similarly for $S_- + Z_0$.

\begin{lemma} \label{lemma:Scover}
    The set $\Om \sans \delcs$ is contained in both
    $S_+ + Z_0$ and $S_- + Z_0$.
\end{lemma}
\begin{proof}
    As the branching foliation is complete in the compact-open topology,
    one can show that the boundary $\del \tilde X$ of $\tilde X = S_+ + Z_0$ 
    is a union of leaves of $\Fcs$.
    (See the proof of Lemma 3.10 in \cite{BBI2} for details.)

    Consider the manifold $\hat M$ defined by the quotient $\tM / Z_0$.
    Then $\tilde X$ quotients down to a subset $\hat X \subof \hat M$
    and $\del \tilde X$ quotients down to $\del \hat X$.
    In particular, $\del \hat X$ is closed subset of $\hat M$
    and the orientation of $\Eu$ shows that
    $\del \hat X$ does not accumulate on itself.

    Let $\hat \Om$ be the quotient of $\Om$ to $\hat M$.
    Its boundary consists of two tori.
    We claim that each torus is either contained in $\hat X$
    or disjoint from $\hat X$.
    Indeed, let $T$ be one of the tori
    and suppose $\hat X \cap T$ is a non-empty proper subset.
    If $T$ is tangent to $\Ecs$,
    then $\hat X \cap T$ is saturated by stable leaves.
    If $T$ is tangent to $\Ecu$,
    then $\hat X \cap T$ is saturated by center leaves.
    In either case, the results in \cref{sec:dimtwo} imply that
    $\del \hat X$ contains a topological line immersed in $T$.
    This line accumulates on itself and gives a contradiction.

    Hence, if $\del \hat X$ intersects the interior of $\hat \Om$,
    it must have a connected component lying entirely in $\hat \Om$.
    This component would be a $cs$-torus,
    which would contradict the assumptions
    given on $\Om$ in \cref{sec:defOm}.
    This shows that $\tilde X$ contains $\Omin$.
    As $\tilde X$ is saturated by stable leaves,
    it also contains any component of $\del \Om$ tangent to $\Ecu$.
\end{proof}
\begin{lemma}
    There is a non-zero linear map $\pi:\bbR^2 \to \bbR$
    such that
    $\pi(z)  \ge  0$ implies $S_+ + z \subset S_+$
    and
    $\pi(z)  \le  0$ implies $S_+ \subset S_+ + z$
    for $z \in Z_0$.
\end{lemma}
\begin{proof}
    This is shown by adapting
    the proofs of Lemmas 3.8 to 3.12
    in \cite{BBI2}.
\end{proof}
The set $\Ku$ defined in \cref{sec:csfiber} is a closed subset of
$\Om \sans \delcs$.
Let $Q$ be the compact set defined by
intersecting $\Ku$ with
a cube of the form $[0,N] \times [0,N] \times I$
for some large $N$.
Since $Z_0$ is a full rank subgroup of $\bbZ^2$,
$N$ may be chosen large enough that
any $x  \in  \Ku$ can be written as
$x = q + z$ with $q  \in  Q$ and $z  \in  Z_0$.

\begin{lemma}
    There is $z_0 \in Z_0$ such that
    $Q \subset S_+ - z_0$ and
    $Q \subset S_- + z_0$.
\end{lemma}
\begin{proof}
    By \cref{lemma:Scover},
    $\{S_+ - z : z \in Z_0 \}$
    is an open cover of $\Ku$, and so some
    finite subset
    \begin{math}
        \{S_+ - z_1, \, \ldots \, , S_+ - z_n\}
    \end{math}
    covers the compact set $Q$.
    Take $z_0$ such that $\pi(z_0)  \ge  \pi(z_i)$ for all $i$.
    The case for $S_-$ is analogous.
\end{proof}
By abuse of notation, if $p=(v,s) \in \bbR^2 \times I$
define $\pi(p)=\pi(v)$.

\begin{lemma}
    There is $r>0$ such that
    $\pi(x) > r$ implies $x \in S_+$ and
    $\pi(x) < -r$ implies $x \in S_-$
    for all $x \in \Ku$. %
\end{lemma}
\begin{proof}
    Choose $r > 0$ such that
    $r - \pi(q) > \pi(z_0)$ for all $q  \in  Q$.
    Any $x \in \Ku$ may be written as
    $x = q + z$ with $q \in Q$ and $z \in Z_0$.
    If $\pi(x) > r$, then $\pi(z-z_0)  \ge  0$ and
    \begin{math}
        x \in Q + z \subset S_+ - z_0 + z \subset S_+.
    \end{math}  \end{proof}
\begin{lemma}
    There is $R > 0$ such that
    if $p,q \in \Ku$
    lie on the same leaf of $\Fcs$,
    then $|\pi(p)-\pi(q)| < R$.
\end{lemma}
\begin{proof}
    Without loss of generality, shift $p$ and $q$ by an element of $Z_0$
    and assume $q \in Q$.
    Let $S'$ be the leaf containing both $p$ and $q$.
    Since $S'$ intersects $S_+ - z_0$ and leaves do not topologically cross,
    $S'$ is disjoint from $S_- - z_0$ and so $\pi(p),\pi(q) > -r-\pi(z_0)$.
    Similarly, $\pi(p),\pi(q) < r+\pi(z_0)$.
    Take $R = 2(r+\pi(z_0))$.
\end{proof}
Up to rescaling $\pi$ so that $R < 1$, this concludes the proof of
\cref{prop:linfoln}.

\bigskip

\acknowledgement
The authors thank Nicolas Gourmelon for helpful conversations.

\bibliographystyle{alpha}
\bibliography{dynamics}

\def\cprime{$'$}
\begin{thebibliography}{RHRHU16}

\bibitem[BBI04]{BBI1}
M.~Brin, D.~Burago, and S.~Ivanov.
\newblock On partially hyperbolic diffeomorphisms of 3-manifolds with
  commutative fundamental group.
\newblock {\em Modern dynamical systems and applications}, pages 307--312,
  2004.

\bibitem[BBI09]{BBI2}
M.~Brin, D.~Burago, and S.~Ivanov.
\newblock Dynamical coherence of partially hyperbolic diffeomorphisms of the
  3-torus.
\newblock {\em Journal of Modern Dynamics}, 3(1):1--11, 2009.

\bibitem[BI08]{BI}
D.~Burago and S.~Ivanov.
\newblock Partially hyperbolic diffeomorphisms of 3-manifolds with abelian
  fundamental groups.
\newblock {\em Journal of Modern Dynamics}, 2(4):541--580, 2008.

\bibitem[Fra70]{Franks1}
J.~Franks.
\newblock Anosov diffeomorphisms.
\newblock {\em Global Analysis: Proceedings of the Symposia in Pure
  Mathematics}, 14:61--93, 1970.

\bibitem[Ham13a]{ham2013leaf}
A.~Hammerlindl.
\newblock Leaf conjugacies on the torus.
\newblock {\em Ergodic Theory Dynam. Systems}, 33(3):896--933, 2013.

\bibitem[Ham13b]{ham-nil}
A.~Hammerlindl.
\newblock Partial hyperbolicity on 3-dimensional nilmanifolds.
\newblock {\em Discrete and Continuous Dynamical Systems}, 33(8):3641--3669,
  2013.

\bibitem[Ham16a]{ham20XXconstructing}
A.~Hammerlindl.
\newblock {Constructing center-stable tori}.
        \newblock {\em Preprint}, 2016.\\
\newblock http://arxiv.org/abs/1610.06290v1.

\bibitem[Ham16b]{ham20XXprox}
A.~Hammerlindl.
\newblock {Properties of compact center-stable submanifolds}.
\newblock {\em Preprint}, 2016.
\newblock http://arxiv.org/abs/1612.03535v1.

\bibitem[HP14]{HP1}
A.~Hammerlindl and R.~Potrie.
\newblock Pointwise partial hyperbolicity in three-dimensional nilmanifolds.
\newblock {\em J. Lond. Math. Soc. (2)}, 89(3):853--875, 2014.

\bibitem[HP15]{HP2}
A.~Hammerlindl and R.~Potrie.
\newblock Classification of partially hyperbolic diffeomorphisms in 3-manifolds
  with solvable fundamental group.
\newblock {\em J. Topol.}, 8(3):842--870, 2015.

\bibitem[HP17]{hp20XXsurvey}
A.~Hammerlindl and R.~Potrie.
\newblock {Partial hyperbolicity and classification: a survey}.
\newblock {\em Ergodic Theory and Dynamical Systems}, 2017.
\newblock To appear.

\bibitem[Pot12]{potrie2012thesis}
R.~Potrie.
\newblock {\em Partial hyperbolicity and attracting regions in 3-dimensional
  manifolds}.
\newblock PhD thesis, 2012.
\newblock arXiv:1207.1822.

\bibitem[Pot15]{pot2015partial}
R.~Potrie.
\newblock Partial hyperbolicity and foliations in {$\Bbb{T}^3$}.
\newblock {\em J. Mod. Dyn.}, 9:81--121, 2015.

\bibitem[PS07]{pujals2007integrability}
E.~R. Pujals and M.~Sambarino.
\newblock Integrability on codimension one dominated splitting.
\newblock {\em Bull. Braz. Math. Soc. (N.S.)}, 38(1):1--19, 2007.

\bibitem[RHRHU11]{rhrhu2011tori}
F.~Rodriguez~Hertz, M.~A. Rodriguez~Hertz, and R.~Ures.
\newblock Tori with hyperbolic dynamics in 3-manifolds.
\newblock {\em J. Mod. Dyn.}, 5(1):185--202, 2011.

\bibitem[RHRHU16]{rhrhu2016nondyn}
F.~Rodriguez~Hertz, M.~A. Rodriguez~Hertz, and R.~Ures.
\newblock A non-dynamically coherent example on {$\Bbb{T}^3$}.
\newblock {\em Ann. Inst. H. Poincar\'e Anal. Non Lin\'eaire},
  33(4):1023--1032, 2016.

\end{thebibliography}

\end{document}